\documentclass[10pt]{amsart}
\usepackage[left=3cm, right=3cm, top = 2.5cm, bottom = 2.5cm ]{geometry}
\usepackage{amsmath,amsfonts,amsthm,bbm, amssymb}
\usepackage[arrow,matrix]{xy}
\usepackage{enumerate}

\usepackage[usenames]{xcolor}
\usepackage{comment}
\usepackage{graphicx}
\usepackage{amssymb}

\usepackage[normalem]{ulem} 
\theoremstyle{plain}
\numberwithin{equation}{section}
\newtheorem{thm}{Theorem}[section]
\newtheorem{cor}[thm]{Corollary}
\newtheorem{lem}[thm]{Lemma}
\newtheorem{prop}[thm]{Proposition}

\newtheorem*{prop*}{\scshape Proposition}
\newtheorem*{thm*}{\scshape Theorem}

\newtheorem*{lem*}{\scshape Lemma}
\newtheorem*{cor*}{\scshape Corollary}

\theoremstyle{definition}
\newtheorem{df}[thm]{Definition}

\newtheorem{rmk}[thm]{Remark}

\providecommand{\frac}[1]{\operatorname{Frac}(#1)}  
\providecommand{\Spec}[1]{\operatorname{Spec}(#1)} 	%
\newcommand{\Pic}{\operatorname{Pic}}		 
\renewcommand{\ker}{\operatorname{Ker}}				 
\renewcommand{\H}{\operatorname{H}}				     
\newcommand{\CH}{\operatorname{CH}}					 
\renewcommand{\lim}{\operatorname{lim}}			 

	\newcommand{\tensor}{\otimes}
	\renewcommand{\cong}{\simeq}
	
	\newcommand{\ol}{\overline}

\def\Xb{\ol{X}}

\newcommand{\cI}{\mathcal{I}}

\newcommand{\cL}{\mathcal{L}}

\newcommand{\cO}{\mathcal{O}}

\usepackage[pdfauthor={Federico Binda}, %
				pdftitle={Torsion 0-cycles with modulus on affine varieties} ,colorlinks=true]{hyperref}
\hypersetup{
	bookmarksnumbered=true,
	linkcolor=black,
	citecolor=blue,
}
\newcounter{elno}   
\newenvironment{romanlist}{
                         \begin{list}{\roman{elno})
                                     }{\usecounter{elno}}
                      }{
                         \end{list}}
                         
\newcounter{elno-abc}   
\newenvironment{listabc}{
                         \begin{list}{\alph{elno-abc})
                                     }{\usecounter{elno-abc}}
                      }{
                         \end{list}}
\begin{document}
\author{Federico Binda}
\title{Torsion 0-cycles with modulus on affine varieties}

\AtEndDocument{\bigskip{\footnotesize%
  (F.~Binda) \textsc{Fakult\"at f\"ur Mathematik,  Universit\"at Regensburg, Universit\"atsstr.~31, 93053 Regensburg, Germany.} 
  \textit{E-mail address},  \texttt{federico.binda@ur.de}}}
\begin{abstract}
In this note, we show that given a smooth affine variety $X$ over an algebraically closed field $k$ and an effective (possibly non reduced) Cartier divisor on it, the Chow group of zero cycles with modulus $\CH_0(X|D)$ is torsion free, except possibly for $p$-torsion if the characteristic of $k$ is $p>0$. This generalizes to the relative setting classical results of Rojtman (for $X$ smooth) and of Levine (for $X$ singular).
\end{abstract}

\maketitle




\setcounter{tocdepth}{1}

\section{Introduction}
\subsubsection{} Let $k$ be an algebraically closed field. If $X$ is a smooth projective variety over $k$, the classical theorem of Rojtman \cite{Rojtman} asserts that the Albanese map $\alpha_{X}\colon \CH_0(X)\to Alb_{X}(k)$ induces an isomorphism of torsion subgroups (modulo $p$-torsion in characteristic $p>0$, the $p$-part being later fixed by Milne in \cite{MilneTorsion}).
Since then, the theorem has been improved in several ways. 

Dropping the projective assumption and replacing the Chow group of $0$-cycles with the Suslin homology group $H_0^{S}(X, \mathbb{Z}) = H_0^{Sing}(X)$, Barbieri-Viale--Kahn \cite{Barbieri-Viale:2010aa} and Geisser \cite{MR3391879} recently proved independently that if $X$ is a reduced normal scheme over $k$, then the Albanese map induces an isomorphism  \[\alpha_X\colon H_0^{Sing}(X)_{tors}\xrightarrow{\sim} Alb_X(k)_{tors}\] up to $p$-torsion groups, refining previous results by {Spie\ss} and Szamuely \cite{SS}.

In a different direction, closer to the original geometric proof, Levine gave  in \cite{MarcTorsion} a generalization of Rojtman's result to singular projective varieties with singularities in codimension $\geq 2$,  showing that the Albanese map induces an isomorphism
\[\alpha_X\colon \CH_0(X, X_{sing})_{tors}\xrightarrow{\sim} Alb(X)_{tors}\]
up to $p$-torsion, where $\CH_0(X, X_{sing})$ denotes the Levine-Weibel cohomological Chow group of $X$ relative to $X_{sing}$, and $Alb(X)$ is here Lang's Albanese variety. The group $\CH_0(X, X_{sing})$ is the quotient of the free abelian group on the set of regular closed points of $X$ by the subgroup generated by cycles that are push-forward to $X$ of divisors of rational functions on integral curves missing $X_{sing}$. If the singularities of $X$ are in codimension $1$, the relations among admissible cycles are more complicated, since one has to allow curves $C$ that meet the singular locus nicely (i.e.~defined by a regular sequence at each point of $C\cap X_{sing}$). In this case, the subgroup of relations $R_0(X,X_{sing})$ is generated by cycles of the form $\nu_{C, *} ({\rm div} f)$ for $C$ a good curve relative to $X_{sing}$ and $f$ a rational function on $C$ that is a unit at each point of $C\cap X_{sing}$.

\subsubsection{}For $X=\Spec{A}$ an affine $k$-variety, Levine proved another generalization of Rojtman result in \cite{Levine3} (see also \cite[Section 7]{SrinivasOverview}). 
\begin{thm}[\cite{Levine3}, Theorem 2.6]\label{thm:LevineRoitman} Let $X$ be an affine $k$-variety over an algebraically closed field $k$. Then the cohomological Chow group $\CH_0(X,X_{sing})$ is torsion free, except possibly for $p$-torsion if the characteristic of $k$ is $p>0$.
	\end{thm} 
As a consequence, he could prove that the cycle class map
\[\CH_0(X, X_{sing})\to K_0(X)\]
is injective, identifying the image with the subgroup $F_0K_0(X)$ of $K_0(X)$ generated by the classes of smooth points of $X$ (this actually requires quite a bit of extra work, since one has to invoke the theory of Chern classes as developed in the unpublished manuscript \cite{MarcBookUnpublished}). Further generalizations of Theorem \ref{thm:LevineRoitman} are due to Srinivas in \cite[Theorem 1]{MR989909} and more recently to Krishna in \cite[Theorem 1.1]{Krishna:2015aa}.
\subsubsection{} Let now $X$ be a quasi-projective smooth $k$-variety equipped with an effective (and possibly non-reduced) Cartier divisor $D$. Let $\CH_0(X|D)$ be the relative Kerz-Saito Chow group of $0$-cycles on $X$ with modulus $D$. This is defined as quotient of the free abelian group of closed points of $X\setminus D$ modulo the subgroup generated by cycles of the form $\nu_{C,*}({\rm div} f)$, where $C$ is an integral curve in $X$, properly meeting $D$, and $f$ is a rational function on $C$ that is congruent to $1$ modulo the pull back to the normalization of $C$ of the divisor $D$.  When $X$ is a projective curve over a field $k$, the group $\CH_0(X|D)$ is by definition the group of divisors on $X \setminus D$ having a fixed trivialization along $D$ and agrees with the group of $k$-rational points of the Rosenlicht-Serre generalized Jacobian ${\rm Jac}(X|D)$ as defined in \cite{SerreGACC}. 

Starting from the two definitions, it's tempting to compare the cohomological relative Chow group of Levine-Weibel with the Chow group of $0$-cycles with modulus of Kerz-Saito. Replacing $X_{sing}$ with $D$, we see that the two Chow groups are quotients of the same set of generators, subject to relations given by rational functions having a prescribed behavior along the restriction to $D$. One of the goals of this note is to show that many techniques developed by Levine and others for the cohomological Chow group can be extended to the modulus setting. More precisely, we prove the following version of what one can call an affine Rojtman theorem with modulus.
\begin{thm}\label{thm:Main-Theorem-Intro}Let $X$ be a smooth affine $k$-variety of dimension at least $2$, $D$ an effective Cartier divisor on it. Then the Chow group of $0$-cycles on $X$ with modulus $D$, $\CH_0(X|D)$, is torsion free, except possibly for $p$-torsion if the characteristic of $k$ is $p>0$.
    \end{thm}
Note that there is a well defined relative cycle class map (for $X$ not necessarily affine, see for example \cite[Theorem 11.6]{BK}, but also \cite[Theorem 4.4.10]{BThesis} for a different approach)
\[cl_{X,D}\colon \CH_0(X|D)\to K_0(X,D),\]
where $K_0(X,D)$ is the relative $K_0$-group of the pair $(X,D)$, that is a particular instance of a more general construction. 
However, without a good theory of relative Chern classes (or Chern classes with modulus) at hand at the moment, we can't deduce directly from our Theorem \ref{thm:Main-Theorem-Intro} the injectivity of the cycle class map $cl_{X,D}$.

In \cite{BK}, the injectivity of $cl_{X,D}$ for $X$ a smooth quasi-projective surface is obtained  as byproduct of a factorization result for the Levine-Weibel Chow group of 0-cycles of the ``double variety'' $S(X,D)$, obtained by glueing two copies of $X$ along the given divisor $D$ (see \cite{BK}, Theorem 1.8). The double construction, together with results of Krishna and Levine on torsion cycles on singular affine varieties recalled above, allows one to prove a stronger version of our Theorem \ref{thm:Main-Theorem-Intro}, encompassing $p$-torsion as well (this is \cite{BK}, Theorem 5.12). 

In this note we use a different (and more direct) argument, strongly inspired by the original works of Levine \cite{Levine3}, and that does not rely on the decomposition Theorem proved in \cite{BK}. In fact, in the proofs we essentially reproduce and adapt Levine's arguments to our context.
This is a continuation of our study of torsion $0$-cycles in \cite{tor-div-rec}.
\section{Zero cycles with modulus}

\subsubsection{}\label{def:Setting-def-Chow-Mod}We briefly recall the definition of the Kerz-Saito Chow group of $0$-cycles with modulus (see \cite{KS1}).
For an integral scheme $\overline{C}$ over $k$ and for $E$ a closed subscheme of $\overline{C}$, we set
\begin{align*}
G(\ol{C},E)
&=\bigcap_{x\in E}\mathrm{Ker}\bigl(\cO_{\overline{C},x}^{\times}\to \cO_{E,x}^{\times} \bigr) \\&= \varinjlim_{E\subset U\subset \overline{C}}\Gamma(U, \ker(\cO_{\overline{C}}^\times \to \cO_{E}^\times)),
\end{align*}
where  $U$ runs over the set of open subsets of $\ol{C}$ containing $E$ (the intersection taking place in the function field $k(\ol{C})^\times$). We say that a rational function $f\in G(\ol{C},E)$ satisfies the modulus condition with respect to $E$.

Let $X$ be a scheme of finite type over $k$, and let $D$ be an effective Cartier divisor on ${X}$. Write $U$ for the complement $X\setminus{D}$ and $Z_0(U)$ for the free abelian group on the set of closed points of $U$. Let $\ol{C}$ be an integral normal curve over $k$ and
let $\varphi_{\ol{C}}\colon \ol{C}\to {X}$ be a finite morphism such that $\varphi_{\ol{C}}(\ol{C})\not \subset D$.   The push forward of cycles along the restriction of $\varphi_{\ol{C}}$ to $C$ gives a well defined group homomorphism
\[
\tau_{\ol{C}}\colon G(\ol{C},\varphi_{\ol{C}}^*(D))\to Z_0(U),
\]
sending a function $f$ to the push forward of the divisor ${\rm div}_{\ol{C}}(f)$.
\begin{df}\label{def:DefChowMod-Definition}

We define the Chow group $\CH_0(X|D)$ of 0-cycles of $X$ with modulus $D$ as the cokernel of the homomorphism 
\[
\tau\colon\bigoplus_{\varphi_{\ol{C}}\colon \ol{C}\to X}G(\ol{C},\varphi_{\ol{C}}^*(D)) \to Z_0(U).
\]
where the sum runs over the set of finite morphisms $\varphi_{\ol{C}}\colon \ol{C}\to X$ with $\ol{C}$ normal and such that $\varphi_{\ol{C}}(\ol{C})\not \subset D$.

\end{df}

\subsection{Relative Picard and cycles with modulus}\label{sec:Picard}
We recall here some basic results about the relative Picard group that we will use later in the text (see \cite[\S 2]{SV}). Let $\overline{X}$ be a scheme and let $Y$ be a closed subscheme of $X$. Let $X$ be the open complement $\ol{X}\setminus Y$ and write $i\colon Y\to \Xb$ for the closed embedding. The relative Picard group $\Pic(\ol{X},Y)$ is the group of isomorphism classes of pairs of the form $(\mathcal{L}, \sigma)$, where $\mathcal{L}$ is a line bundle on $\Xb$ and $\sigma$ is a trivialization of $\mathcal{L}$ along $Y$, i.e.~the datum of an isomorphism $\sigma\colon \mathcal{L}_{|Y}\xrightarrow{\simeq} \cO_{Y}$, with multiplication given by the tensor product of line bundles.

There is an exact sequence
\begin{equation}\label{eq:fundsequencePic}\Gamma(\Xb, \cO^\times)\to \Gamma(Y, \cO^\times)  \to \Pic(\Xb, Y) \to \Pic(\Xb)\to \Pic(Y)\end{equation}
and we can identify $\Pic(\Xb, Y)$ with the cohomology group $\H^1_{\rm Zar}(\Xb, \cO^\times_{\Xb|Y})$, where $\cO^\times_{\Xb|Y}$ denotes the kernel of the natural surjection of sheaves of units $\cO_{\Xb}^\times\to i_* \cO_Y^\times$. 

\subsubsection{}\label{sec:Miscellanea-Pic} Assume that $\Xb$ is integral and Noetherian and that $Y$ has an affine open neighborhood in $\Xb$ (this happens e.g.~when $\Xb$ is a curve). Let ${\rm Div}(\Xb, Y)$ be the group of Cartier divisors on $\Xb$ whose support does not intersect with $Y$. Note that if $X$ is regular, then ${\rm Div}(\Xb, Y)$ is a free abelian group generated by irreducible divisors $T\subset X$ which are closed in $\Xb$ (this is \cite[Lemma 2.4]{SV}). By \cite[2.3]{SV}, one has the following exact sequence
\begin{equation}\label{eq:sequence-Pic-Cartier} 0\to \Gamma(\Xb,\cO^\times_{\Xb,Y}) \to G(\Xb, Y) \to {\rm Div}(\Xb, Y) \to \Pic(\Xb, Y) \to 0.\end{equation}
If $Y$ does not admit an affine open neighborhood, the last map in the displayed sequence is no longer surjective: its image agrees with the subgroup of $\Pic(\Xb, Y)$ consisting of liftable elements. See \cite[3.2]{KSY}.

We need the following auxiliary Lemma   to relate the relative Picard group with the Chow group with modulus (in the case of curves).
\begin{lem}\label{lem:lemmautile}Let $f\colon \ol{C}'\to \ol{C}$ be a proper surjective morphism of integral $k$-schemes of finite type and let $E\subset \ol{C}$ be an effective Cartier divisor. Then, for $E'=f^* E$, we have an injection
	\[f^*\colon G(\ol{C}, E)\hookrightarrow G(\ol{C}', E').\]
	\begin{proof} By \cite[Lemma 2.7.2]{KSY}, the system $\{ f^{-1}(U) \,|\, U\supset E, U \, {\rm open}\}$ is cofinal among the open neighborhoods of $f^{-1}(E)$ in $\ol{C}'$. Thus, it's enough to show that the natural map $\cO_{U|E}^\times\to f_* \cO^\times_{f^{-1}(U)|E'}$ is injective, for $U$ running on the set of open neighborhoods of $E$ in $\ol{C}$. Since $f$ is surjective, we have $\cO_{\ol{C}}^\times\hookrightarrow f_* \cO_{\ol{C}'}^\times$, whence the claim follows.
		\end{proof}
	\end{lem}
\subsubsection{}\label{sec:construction-Pic-Chow}
Let $\ol{C}$ be an integral curve over $k$ and let $E$ be an effective Cartier divisor on $\ol{C}$. Write $C = \ol{C}\setminus E$. Then we have a map ${\rm Div}(\ol{C}, E) \to Z_0(C)$, obtained by sending a Cartier divisor  to its associated Weil divisor (whose support is, by definition, disjoint from $E$). Let $\varphi\colon \ol{C}^N\to \ol{C}$ be the normalization morphism. According to Definition \ref{def:DefChowMod-Definition}, we have the following presentation of the Chow group of $0$-cycles with modulus on $\ol{C}$
\[ G(\ol{C}^N, \varphi^* E) \xrightarrow{\varphi_*} Z_0(C)\to \CH_0(\ol{C}|E)\to 0.\]
Combining Lemma \ref{lem:lemmautile} (for $f=\varphi$) and \eqref{eq:sequence-Pic-Cartier} we obtain then a natural map ${\rm Div}\colon \Pic(\ol{C}, E) \to \CH_0(\ol{C}|E)$ which is, in general, not surjective nor injective. If $\ol{C}$ happens to be normal, it is classically known that this map is an isomorphism (see e.g.~\cite{SerreGACC}): if $\ol{C}$ is moreover projective, we can identify the degree-$0$-part of $\CH_0(\ol{C}|E)$ with the group of $k$-rational points of ${\rm Jac}(\ol{C}|E)$, the (connected component of the identity of the) Rosenlicht-Serre generalized Jacobian of the pair $(\ol{C}, E)$ (see also \cite[9.4.1]{KSY}).  

If $C$ is normal (i.e.~$\ol{C}_{\rm sing} \subset |E|$), then \cite[Lemma 2.4]{SV} recalled above in \ref{sec:Miscellanea-Pic} implies that the map ${\rm Div}$ is  surjective. In this case, we immediately see that $\CH_0(\ol{C}|E) = \Pic(\ol{C}^N, \varphi^* E)$ and that we are describing nothing but the canonical surjection
\[\Pic(\ol{C}, E)\to \Pic(\ol{C}^N, \varphi^* E).\]

Its kernel is isomorphic to the quotient $K_{\ol{C}|E} := G(\ol{C}^N, \varphi^* E)/G(\ol{C}, E)$. Note that we have an exact sequence of sheaves on $\ol{C}$
\[ 0\to \cO^\times_{\ol{C}|E} \to \varphi_* \cO^\times_{\ol{C}^N| \varphi^* E} \to \mathcal{S}\to 0 \]
where $S$ is a skyscraper sheaf, supported on $|E|\cap \ol{C}_{\rm sing}$, and $K_{\ol{C}|E} = \varinjlim_{U\supset |E| \cap \ol{C}_{\rm sing}}\Gamma(U,S) = \Gamma(\ol{C}, S)$. 

In order to compute  $\Gamma(\ol{C}, S)$ we can assume that $\ol{C}$ is projective (if not, compactify $\ol{C}\subset \ol{C}'$ so that $\ol{C}'_{\rm sing} = \ol{C}_{\rm sing}$). Factor the map $\varphi\colon \ol{C}^N\to \ol{C}$ as $\pi_1\colon \ol{C}^N\to Y$ followed by $\pi_2\colon Y\to \ol{C}$, where $Y$ is the reduced and projective curve with ordinary  singularities associated to $\ol{C}$. Write $T$ for the kernel of $\Pic(Y, \pi_1^* E) \to \Pic(\ol{C}^N, \varphi^*E)$ and $L$ for the kernel of $\Pic(\ol{C}, E)\to \Pic(Y, \pi_1^* E)$. Then we have a canonical exact sequence
\[ 0\to L\to  K_{\ol{C}|E}\to T\to 0.\]
If we now suppose now that $k$ is algebraically closed, and that $n$ is an integer prime to the characteristic of $k$, we can follow the argument of \cite[7.5.18]{Liu} to see that both $L$ and $T$ are $n$-divisible (and so, $K_{\ol{C}|E}$ is $n$-divisible as well). 

\subsubsection{}\label{sec:map-Pic-Chow-X}Let $X$ be a scheme of finite type over $k$  and $D$  an effective Cartier divisor on it. Let $\nu\colon C\hookrightarrow X$ be an integral curve, properly intersecting $D$ (i.e.~such that $\nu(C)\not\subset D$). The divisor $D$ restricts then to an effective divisor $\nu^* D$ on $C$, and we have a well-defined push forward map
\[\nu_*\colon \CH_0(C|\nu^*( D)) \to \CH_0(X|D).\]
By composing it with the canonical map constructed in \ref{sec:construction-Pic-Chow}, we obtain a map, still denoted $\nu_*$, $\Pic(C, \nu^*( D))\to \CH_0(X|D)$, that we will use repeatedly in the text. 

 \subsection{The rigidity theorem}

Let ${X}$ be a quasi-projective variety  over an algebraically closed field $k$ of exponential characteristic $p\geq1$. Let $D$ be an effective Cartier divisor on $X$. Write $U$ for the open complement $X\setminus D$. Let $C$ be a smooth curve over $k$ and let $W$ be a $1$-cycle on $C\times X$ that is flat over $C$ and such that $|W|\subset C\times U$. Let $x$ be a closed point in $C$. Since $W$ is flat over $C$, $\dim(W\cap (x\times U)) = 0$, so that $W$ and $x\times U$ are in good position. We denote by $W(x)$ the cycle
\[W(x) = p_{2, *} (W \cdot  ( \{x\}\times U)).\]
It is a $0$-cycle on $X$, supported outside $D$. Write $[W(x)]$ for its class in $\CH_0(X|D)$. The following result improves \cite[Theorem 2.13]{tor-div-rec}, using the same argument of \cite[Proposition 4.1]{LW}.

\begin{thm}\label{thm:rigidity}
In the above notations, let $n$ be an integer prime to $p$.
  Assume that there exists a dense open subset $C_0$ of $C$ such that for every $x\in C_0(k)$ one has
	\[n\cdot [W(x)] =0 \quad \text{ in } \CH_0(X|D) \]
	Then the function $x\in C(k)\mapsto [W(x)]$ is constant.
\end{thm}


\section{Torsion cycles on affine varieties}
\subsection{An easy moving} Throughout this section, we fix an algebraically closed field $k$. We discuss a result that concerns the image of torsion cycles from curves to affine varieties with modulus. As the reader will soon notice, we owe a great intellectual debt to \cite{Levine3}. 
We start by recalling the following result (see \cite{Levine3}, Corollary 1.2).
\begin{lem}\label{lem:good-compactification} Let $X$ be a quasi-projective variety, smooth over $k$. Let $D$ be an effective Cartier divisor on $X$ and let $\nu\colon C\hookrightarrow X$ be a reduced curve in $X$, properly intersecting $D$. Then there exists a projective closure $\Xb$ of $X$ such that, if $\ol{C}$ (resp.~$\ol{D}$) denotes the closure of $C$ (resp.~$D$) in $\Xb$,  then the following hold:
\begin{enumerate}
\item $\ol{D}\cap \ol{C} = C\cap D$,
\item $\Xb \setminus \ol{D}$ is normal. In particular, the singular locus $\Xb_{sing}$ ($\subset \Xb - X$ ) has codimension at least $2$ on $\Xb$ at each point of $\Xb_{sing}\cap (\ol{C} \setminus C)$. 
\end{enumerate}
\end{lem}

The following Proposition is an application of a classical moving argument for $0$-cycles on smooth varieties. Its proof is representative of the arguments   used repeatedly in \cite{BK}, Sections 5 and 6 (see, in particular, Lemma 5.4 in \textit{loc.cit.}).
\begin{prop}\label{prop:key-prop} Let $X$ be an affine smooth $k$-variety of dimension at least $2$ and let $D$ be an effective Cartier divisor on it. Let $\nu\colon C \hookrightarrow X$ be an integral curve,  properly intersecting $D$.
	Write $\nu_{*}$ for the push-forward map (see \ref{sec:map-Pic-Chow-X})
\[\nu_{*} \colon  \Pic(C, \nu^*(D)) \to \CH_0(X|D).\]
Let $n$ be an integer prime to the characteristic of $k$ and let $\alpha$ in  $\Pic(C, \nu^*(D))$ be an $n$-torsion class. Then $\nu_{*}(\alpha) = 0$ in $\CH_0(X|D)$. 
\end{prop}
\begin{proof}
Let $\Xb$ be a compactification of $X$ satisfying conditions (1) and (2) of Lemma \ref{lem:good-compactification}. We can find an integral projective surface $\ol{Y}\subset \Xb$ satisfying 
\begin{romanlist}
\item $\ol{Y} \supset \ol{C}$;
\item $\ol{Y}$ intersects $\ol{D}$ properly;
\item $Y = \ol{Y} \cap X$ is  affine. The divisor $D_Y = Y\cap D$ is an effective Cartier divisor on $Y$.
\end{romanlist}
The condition that $\overline{D}$ intersects properly $\overline{Y}$ is automatic for general $\overline{Y}$, since we can always assume that $\overline{Y}$ does not contain the generic points of $D$.  Since $X$ is smooth by assumption, the Bertini theorem of Altman and Kleiman  \cite[Theorems 1 and 7]{AK} guarantees that we can find $\overline{Y}$ containing $\overline{C}$ and such that the affine surface $Y$ is smooth away from $C$. Since   $\overline{Y}$ is constructed by intersecting independent general hypersurface sections of $\Xb$,   we see that the affine part $Y$ is a complete intersection on $X$. In particular, it is Cohen-Macaulay. 

By condition (2) of Lemma \ref{lem:good-compactification}, we can moreover assume that $\ol{Y}$ is such that every point of $\ol{C} \setminus C$ is either a smooth point of $\ol{Y}$ or an isolated singularity. By resolving the singularity of $\ol{Y}$ that are on $\ol{C}\setminus C$ and the singularities of $\ol{C} \setminus C$, we can actually assume that every point $x\in \ol{C}\setminus C$ is a regular point of $\ol{Y}$ and that $\ol{C}$ is also regular at $x$. Finally, again by condition (2) of Lemma \ref{lem:good-compactification} and by what we just said, we can also assume that $\overline{Y}$ is normal at each point of $(\overline{Y}\setminus Y)\setminus D_Y$. In particular, $\overline{Y}$ is Cohen-Macaulay except possibly for the points of $B = D_Y\cap (\overline{Y}\setminus Y)$ (recall that a normal local ring is $S_2$ by Serre's criterion, and since $\overline{Y}$ is two dimensional this implies being Cohen-Macaulay). By resolving the singularities of $\ol{Y}$ that are on $B$, we can then assume that $\ol{Y}$ itself is a Cohen-Macaulay surface.

Replacing $X$ with $Y$, we are reduced to the following case
\begin{listabc}
\item $X$ is an integral affine surface,  smooth away from $C$, and $D$ is an effective Cartier divisor on it. 
\item $\ol{X}$ is a projective compactification of $X$, and is a Cohen-Macaulay surface.
\item $\nu\colon C\hookrightarrow X$ admits a compactification $\ol{\nu}\colon \ol{C}\to  \Xb$ such that $\ol{C}$ is regular at $\ol{C}\setminus C$ and such that $\ol{C}\cap \ol{D} = C\cap D$.  The surface $\ol{X}$  is regular at every point of $\ol{C}\setminus C$.
\end{listabc}
Let $F$ denote the closed complement $\Xb \setminus X$. Since $X$ is affine, $F$ is the support of a very ample line bundle $\cL$ on $\Xb$. For $d$ sufficiently large, we can find global sections \[s_0\in \H^0(\Xb, \cL^{\tensor d}\tensor \cI_{\ol{C}}) \subset \H^0(\ol{X}, \cL^{\tensor d}) \quad \text{and } s_\infty \in \H^0(\Xb, \cL^{\tensor d})\] such that $(s_0)=W_0 = \ol{C} \cup E$, for $E$ integral, intersecting $\ol{D}$ properly and away from $C\cap D$, is a reduced connected curve, and such that $(s_\infty) = W_\infty$ is contained in $\Xb \setminus X = F$.  

Indeed, we can find sections $a_0, \ldots, a_m$ of $V=\H^0(\Xb, \cL^{\tensor d}\tensor \cI_{\ol{C}})$ such that the rational map $\psi\colon \Xb\to \mathbb{P}^m$ defined by $(a_0:a_1:\ldots a_m)$ is a locally closed immersion on $\Xb\setminus \ol{C}$. In particular, it is separable. Since $\Xb$ is integral and $\ol{C}$ is reduced, the general divisor $K$ in $V$  is generically reduced, and if $K'$ denotes the closure of $K\setminus \ol{C}$, then $K'$ is irreducible and reduced (this is classical Bertini, since $\overline{X}\setminus \ol{C}$ is reduced and irreducible).  Moreover, since $\Xb$ is Cohen-Macaulay, $K$ is itself a Cohen-Macaulay scheme. But a generically reduced Cohen-Macaulay scheme is in fact reduced by e.g.~\cite[Prop. 14.124]{GW} (see also  Lemma 3.1 at page 114 in \cite{LW} on the existence of enough pure Cartier curves, where the same argument is used). Hence $W_0$ above can be taken to be reduced and $E$ can be taken to be integral. The condition that the extra component $E$ intersects $D$ properly and away from $C\cap D$ is clearly open on the space of sections.

 Using $s_0$ and $s_\infty$, we can define a pencil $P = \{W_t \,|\, t\in \mathbb{P}^1\}$ of hyperplane sections of $\Xb$, interpolating between $W_0$ and $W_\infty$. More precisely, let $W$ be the flat cycle
\[W\subset \Xb\times \mathbb{P}^1 \]
given by the equation $s_0 + t s_\infty$ for $t$ a rational coordinate on $\mathbb{P}^1$. Then for general $t$, $W_t$ is  integral, intersects  $\ol{D}$ properly and misses the singular locus of $\ol{D}_{red}$.

Let $S$ be the spectrum of the local ring of $\mathbb{P}^1$ at $0$, $s$ its closed point, $\eta$ its generic point.
We denote by $\pi_S\colon W_S\to S$ the base change of $W\to \mathbb{P}^1$ to $S$. By construction, the special fiber $(W_S)_s$ coincides with $W_0$, while the generic fiber $W_\eta \to k(t)$ represents the generic member of the pencil. The family $W_S\to S$ is flat, projective, so  $\chi(\cO_{W,t}) =\chi(\cO_{W,s}) = \chi(\cO_{W_0}) $. Since $W_t$ is integral and $W_0$ is reduced and connected, we conclude that the curves in the family have constant arithmetic genus $g = p_a(W_t)$. Hence, the morphism $\pi_S$ is cohomologically flat in dimension $0$ (see \cite[page 206]{NeronModels}). By Artin's representability theorem (\cite[Theorem 8.3.1]{NeronModels}), the relative Picard $\mathbf{Pic}^0_{(W_S | D_S) / S} \to S$ is representable by a (locally finitely presented) algebraic space over $S$,  where $D_S$ is the base change of $\ol{D}$ to $W_S$ (and is an horizontal divisor on $W_S$). Actually, Artin's theorem shows the representability of  $\mathbf{Pic}^0_{W_S / S} \to S$, but $\mathbf{Pic}^0_{(W_S | D_S) / S} \to S$ is a torsor over  $\mathbf{Pic}_{W_S / S} \to S$ for the group $G = \pi_{D,*}\mathbb{G}_{m,D}$. Thus $\mathbf{Pic}^0_{(W_S | D_S) / S} \to S$ is representable by an algebraic space as well (see also \cite[Theorem 5.2]{RaynaudRep}). Since $h^0(\cO_{W_\eta}) = h^0(\cO_{W_0})$, by \cite[Theorem 4.1.1, Proposition 8.0.1]{RaynaudRep}, $\mathbf{Pic}^0_{W_S / S}$ coincides with its maximal separated quotient, and thus it is representable by a scheme, separated and locally of finite type over $S$. By \cite[Lemma 3.6]{Gro}, the same holds for  $\mathbf{Pic}^0_{(W_S | D_S) / S}$.

 Since the canonical restriction map
\[\Pic(\ol{C},\nu^*(D) = \ol{\nu}^*(\ol{D})) \xrightarrow{j^*} \Pic(C,\nu^*D) \] 
is surjective, as $\ol{C} \cap \ol{D} = C\cap D$, we can lift $\alpha$ to a cycle $\tilde{\alpha}$ in $\Pic(\ol{C},\nu^*D)$. By assumption, we have that $n \alpha = 0$ in $\Pic(C,\nu^*D)$, so that $n\tilde{\alpha} \in \ker(j^*)$, that is the subgroup of $\Pic(\ol{C},\nu^*D)$ that is generated by classes of  (regular) points in $\ol{C}\setminus C$. Since $W_0 = \ol{C} \cup E$, we have a canonical map
\begin{equation}\label{eq:relative-pics}\pi^*\colon \Pic(W_0, W_0\cap \ol{D}) = \Pic(W_S \tensor k(s) , D_S\tensor k(s)) \to \Pic(\ol{C}, \nu^*D) \times  \Pic(E, E\cap \ol{D}) \end{equation}
induced by $\pi\colon \ol{C}\amalg E\to W_0$, that is surjective. Note that the relative Picard groups in \eqref{eq:relative-pics} are well-defined, even if $E\cap D$ or $W_0\cap \ol{D}$ are not Cartier divisors or if $E$ and $W_0$ are not regular (see \ref{sec:Picard}). By \eqref{eq:fundsequencePic}, we have the following exact sequences:

\begin{align} &\Gamma(W_0, \cO^\times) \to \Gamma(W_0\cap \ol{D}, \cO^\times)\to \Pic(W_0, W_0\cap \ol{D}) \to \Pic(W_0)\to0 \\ \nonumber
	&k^\times\to \Gamma(\nu^*D, \cO^\times)\to \Pic(\ol{C}, \nu^*D) \to \Pic(\ol{C})\to 0\\\nonumber
	& k^\times \to \Gamma(E\cap \ol{D},\cO^\times)\to  \Pic(E, E\cap \ol{D}) \to \Pic(E)\to 0.
	\end{align}
Since $E\cap \ol{D}\cap \ol{C} =\emptyset$, we have that $\Gamma(W_0\cap \ol{D}, \cO^\times) =\Gamma(\nu^*D, \cO^\times) \times  \Gamma(E\cap \ol{D},\cO^\times)$, and therefore the kernel of $\pi^*$ in \eqref{eq:relative-pics} coincides with the kernel of the map (still induced by $\pi$), $\Pic(W_0) \to \Pic(\ol{C}) \times \Pic(E) $. Let $\tilde{W_0}$ (resp.~$\tilde{\ol{C}}$ and $\tilde{E}$) be the normalization of $W_0$ (resp.~of $\ol{C}$ and $E$). Clearly, $\tilde{W_0} = \tilde{C}\amalg \tilde{E}$. We have then canonical exact sequences 
\begin{align*}&0\to K\to \Pic(W_0)\to \Pic(\tilde{\ol{C}})\times \Pic(\tilde{E})\to 0 \\&0\to L\to  \Pic({\ol{C}})\times \Pic({E}) \to \Pic(\tilde{\ol{C}})\times \Pic(\tilde{E})\to 0\end{align*}
and the groups $K$ and $L$ can be identified with the groups of $k$-rational points of two affine, commutative, and connected algebraic groups (namely, the affine parts of the generalized Jacobian varieties of the singular curves $W_0$ and $\ol{C}\amalg E$). The kernel $U = \ker (\pi^*)$ is then (the group of $k$-points of) the kernel of the map $\pi\colon K\to L$ between two affine and commutative algebraic groups. It is therefore itself a commutative affine algebraic group over $k$, and so it's group of points is $n$-divisible, being $n$ coprime with the characteristic of $k$.

We choose now $S'\to S$  a DVR dominating $S$ so that there is a section
\[\gamma'\colon S' \to \mathbf{Pic}^0_{(W_{S'} | D_{S'}) / {S'}}  \]
satisfying $\gamma'(s') = n \beta_0$, where $\beta_0$ lifts to  $\Pic(W_{0,S'}, (W_0\cap \ol{D})\times_S S')$ the element $(\tilde{\alpha}, 0)$ of \[ \Pic(\ol{C}, \nu^*D) \times  \Pic(E, E\cap D).\] Here $s'$ denotes the closed point of $S'$, above $0\in \mathbb{P}^1$. Note that, since the class of $n\tilde{\alpha}$ in $\Pic(\ol{C}, \nu^*D)$ is represented by points in $\ol{C}\setminus C$, we can further assume that the divisor $Z'$ on $W_S\times_S S'$ representing $\gamma'$ is supported on $F_{S'}$.

Write \[n_{\mathbf{Pic}}\colon \mathbf{Pic}^0_{(W_{S'} | D_{S'})}  \to \mathbf{Pic}^0_{(W_{S'} | D_{S'})} \]
for morphism of schemes given by the multiplication by $n$ on  $ \mathbf{Pic}^0_{(W_{S'} | D_{S'})}$. 
Let  $T$ be the normalization of an irreducible component of $n_{\mathbf{Pic}}^{-1}(\gamma'(S')) \subset \mathbf{Pic}^0_{(W_{S'} | D_{S'})} $ passing through $\beta_0$ and let $Z$ be the divisor on $W_{S'}\times_{S'} T$ representing $\gamma\colon T\to \mathbf{Pic}^0_{(W_{S'} | D_{S'})}$. Write $p\colon T\to S$ for the composite map. For $t\in T$, let $Z(t)$ denote the divisor on $W_{p(t)}$ given by
\[ Z(t) =  p_{1,*} (Z\cdot (p_2^*(t) \cap \{t\}\times W_{S'}))\]
where $p_1$ and $p_2$ are the two projections from $W_{S'} \times_{S'} T$. Note that, since both $\gamma'$ and $\gamma$ are defining subschemes of the relative Picard scheme $\mathbf{Pic}^0_{(W_{S'} | D_{S'})}$ of line bundles with a trivialization along $D_{S'}$, we have that 
\[Z \subset  W_{S'}\times_{S'} T \hookrightarrow W \times_{\mathbb{P}^1} T \hookrightarrow \ol{X} \times T \]
is supported away from $\ol{D} \times T$. Taking $T$ smaller if necessary, we can also assume that $Z$ is actually closed in $X\times T$.  Moreover, the choice of $T$ gives that 
\[ n Z(t) = Z'(t), \quad \text{ in } \Pic(W_{p(t)}, D_{p(t)})\]
(note that for $t$ in a dense subset of $T$, $W_{p(t)}$ is actually smooth over $k$ and $D_{p(t)}$ is an effective Cartier divisor on it). Pushing forward to $X$ gives then a map
\[Z\colon T(k)\to \CH_0(X|D) \]
and since $Z'(t)$ is supported on $\ol{X}\setminus X$ for every $t$, $n Z(t) =0$ in $\CH_0(X|D)$. In other words, we have a family of $n$-torsion $0$-cycles on $X$ with modulus $D$, parametrized by $T$ and represented by $Z$. By the Rigidity Theorem \ref{thm:rigidity}, the family is constant. In particular,  whenever $T$ is realized as dense open subset of any curve ${T}^*$ such that $Z$ extends to 
\begin{equation}\label{eq:1-key-lemma} Z^*\colon T^*(k)\to \CH_0(X|D)\end{equation}
then for every $t\in  T^*(k)$, the class of $Z^*(t)$ will be the same, since Theorem \ref{thm:rigidity} only requires to check what happens on a dense set. 
Let now $\ol{T}$ be a smooth projective model of $T$, containing $T$ as open dense subset. Let $\beta_\infty$ a point of $\ol{T} \setminus T$ above $\infty \in \mathbb{P}^1$ and take $T^*$ to be $T\cup \{\beta_\infty\}$. Then since $Z\cap (X\times \beta_\infty) \subset W_\infty \subset \ol{X} \setminus X$, $Z$ is still closed in $X\times T^*$, it defines a map as in \eqref{eq:1-key-lemma}.
In particular, since $Supp(Z^*(\beta_\infty)) = Z\cap X\times \beta_\infty = \emptyset$, $Z^*(\beta_\infty) = 0$ in $\CH_0(X|D)$. Thus
\[\nu_*(\alpha) = Z(\beta_0) = Z^*(\beta_\infty) = 0\quad \text{ in } \CH_0(X|D)\]
completing the proof.
\end{proof}

The following lemma is easy to check. 
\begin{lem}\label{lem:lem-blow-up}Let $X$ be an affine smooth $k$-variety, of dimension at least 2 and let $D$ be an effective Cartier divisor on it. Let $u\colon X'\to X$ be a sequence of blow-ups with center in points lying over $X\setminus D$. Then one has an isomorphism
\[u_*\colon \CH_0(X'|D) \xrightarrow{\simeq} \CH_0(X|D).\]
\end{lem}
We then have the following Corollary to Proposition \ref{prop:key-prop}.
\begin{cor}\label{cor:torsion-zero-blow-up} Let $X, X', D$ and $u$ be as in Lemma \ref{lem:lem-blow-up}. Let $\nu\colon C \hookrightarrow X'$ be an integral curve in $X'$, properly intersecting $D$. Assume that $C\setminus |\nu^*(D)|$ is normal. 
Write $\nu_{*}$ for the push-forward map
\[\nu_{*} \colon \Pic(C, \nu^*(D))   \twoheadrightarrow\CH_0(C |\nu^*(D))  \to \CH_0(X'|D)\xrightarrow{\simeq} \CH_0(X|D).\]
Let $n$ be an integer prime to the characteristic of $k$ and let $\alpha$ in  $\Pic(C, \nu^*(D))$ be an $n$-torsion class. Then $\nu_{*}(\alpha) = 0$ in $\CH_0(X'|D)$ (and a fortiori in $\CH_0(X|D)$). 
\begin{proof}Let $E$ be the exceptional divisor of the blow-up $u\colon X'\to X$. We have to consider two cases. If $\nu(C)\subset E$, then $C$ is a projective curve and any cycle $\alpha \in \CH_0(C |\nu^*(D))$ of degree zero satisfies $u_*(\alpha) =0$. In particular, this applies to torsion cycles. If $\nu(C)$ is not completely contained in the exceptional divisor, write $Z$ for the image $u(C)\subset X$ and write $j\colon Z\to X$ for the inclusion. Since the center of the blow-up is disjoint from $D$, $u$ induces an isomorphism on a neighborhood of $\nu^*(D)$ between $C$ and $Z$. In particular, $\nu^*(D) \cong j^*(D)$. 
	We have a commutative diagram
    \[ \xymatrix{ \Pic(C ,\nu^*(D))\ar[r]^{\nu_*} & \CH_0(X'|D)\ar[d]^{u_*}\\ 
    \Pic(Z, j^*(D))\ar[u]^{u^*}\ar[r]^{j_*} &\CH_0(X|D)
    }
    \]
where we keep writing $u$ for the induced morphism $C\to Z$. The morphism $u^*$ is surjective and the kernel is $n$-divisible (in fact, the kernel of $u^*$ agrees with the kernel of $\Pic(Z)\to \Pic(C)$ since $\nu^*(D) \cong j^*(D)$, and this can be seen to be $n$-divisible as discussed in the proof of Proposition \ref{prop:key-prop} or by an explicit cohomological computation). In particular, for every torsion class $\alpha \in   \Pic(C ,\nu^*(D))[n]$, we can find $\beta\in \Pic(Z,j^*(D))[n]$ such that $u^*(\beta)= \alpha$. But then we have $0=j_*(\beta)= u_*(\nu_*(\alpha))$ (by Proposition \ref{prop:key-prop}), from which we conclude $\nu_*(\alpha) =0$ since $u_*$ is an isomorphism.
    \end{proof}
\end{cor}

\begin{rmk}\label{rmk:Pic-and-Chow-divisibility-Corollary}Note that, since $C\setminus |\nu^*(D)|$ is normal by assumption, the kernel $K_{\ol{C}|E}$ of the natural surjection $\Pic(C, \nu^*(D))   \twoheadrightarrow\CH_0(C |\nu^*(D))$ is $n$-divisible (see \ref{sec:construction-Pic-Chow}). Hence, there is a surjection $\Pic(C, \nu^*(D))[n]   \twoheadrightarrow\CH_0(C |\nu^*(D))[n]$ between the $n$-torsion subgroups, i.e.~we can represent each $n$-torsion $0$-cycle by a class in the relative Picard group. A direct consequence of  Corollary  \ref{cor:torsion-zero-blow-up} is then that any $n$-torsion $0$-cycle on $C$ with modulus $\nu^*(D)$ vanishes in $\CH_0(X'|D)$.
	\end{rmk}
\subsection{Vanishing results}We now come to the main results of this note.

\subsubsection{}\label{def:definition-n-1}We continue with the above notations, so $X$ is a smooth affine variety over an algebraically closed field $k$, $D$ is an effective Cartier divisor on $X$ and $X'\to X$ is obtained from $X$ by a sequence of blow-ups with centers lying over $X\setminus D$. 

We will say that an integral curve $\nu\colon C\hookrightarrow X'$ contained in $X'$ (or in $X$) is a good curve with respect to $D$ if the following conditions are satisfied:
\begin{listabc}
	\item The curve $C$ intersects $D$ properly, so that $\nu^*(D)$ is an effective Cartier divisor on $C$.
	\item The open complement $C\setminus |\nu^*(D)|$ is normal.
	\end{listabc}
For a good curve $C$ in $X$ or in $X'$, we have, thanks to Proposition \ref{prop:key-prop}, Corollary \ref{cor:torsion-zero-blow-up} and Remark \ref{rmk:Pic-and-Chow-divisibility-Corollary}, that the pushforward map $\nu_*$ induces the zero map
\begin{equation}\label{eq:zero-map-torsion}\nu_*\colon \CH_0(C|\nu^*(D))[n] \xrightarrow{0} \CH_0(X|D)\quad  \text{( or $\nu_*\colon \CH_0(C|\nu^*(D))[n] \xrightarrow{0} \CH_0(X'|D)$) }\end{equation}
for $n$ an integer prime to the characteristic of $k$.
\subsubsection{}
Let $\alpha$ be a $0$-cycle with modulus   on $C$, supported on $C \setminus \nu^*D$, giving a class $\alpha \in \CH_0(C|\nu^*D)$. Since $C$ is good, $\alpha$ is $n$-divisible in $\CH_0(C|\nu^*D)$ (since the group $\Pic(C, \nu^*(D))$ is $n$-divisible, and $\CH_0(C|\nu^*D)$ is a quotient of it). Then there is a $0$-cycle $\beta$ on $C$ such that $n\beta = \alpha$ in $\CH_0(C|\nu^*D)$. 
We define the element $n^{-1}_C(\alpha) \in \CH_0(X'|D)$ to be the class of $\nu_*(\beta)$. By  \eqref{eq:zero-map-torsion}, the class is well defined.  

\begin{rmk} Let $(\Xb,D)$ be a pair consisting of a smooth and projective $k$-variety $\Xb$ and an effective Cartier divisor $D$  on it. Let $T_{\Xb|D}$ denote the subgroup of the group of degree zero $0$-cycle $\CH_0(\Xb|D)^0$ that is generated by the images of torsion $0$-cycles on proper smooth curves mapping to $\Xb$ (and having image not contained in $D$). In \cite{tor-div-rec}, Section 2, we proved that the operation $n^{-1}$ as defined in \ref{def:definition-n-1} satisfies two important properties (the second implied by the first)
\begin{romanlist} 
\item given two smooth curves $C_1$ and $C_2$ in $\Xb$ such that $\alpha\in C_1\cap C_2$, one has the equality \[n_{C_1}^{-1}(\alpha) = n_{C_2}^{-1}(\alpha)\quad  \text{ in $\CH_0(\Xb|D)^0/T_{\Xb|D}$ }\]  
\item there is a well defined map
\[n_X^{-1}\colon \CH_0(\Xb|D)^0/T_{\Xb|D} \to \CH_0(\Xb|D)^0/T_{\Xb|D}\]
such that, for every $\alpha \in \CH_0(\Xb|D)^0/T_{\Xb|D}$, we have $n_X^{-1}(n\alpha) = \alpha$ in $\CH_0(\Xb|D)^0/T_{\Xb|D}$.
\end{romanlist}
The projectivity of $\Xb$ was used in an essential way to prove the two statements. Replacing $\Xb$ with $X$ affine, we have to modify slightly the arguments using the same strategy as in the proof of \ref{prop:key-prop}, making use of the nice compactification provided by Lemma \ref{lem:good-compactification}. 
Again, we closely follow \cite{Levine3} (or \cite{MarcTorsion}, as we did in \cite{tor-div-rec}) to prove the following Lemma.
\end{rmk}
\begin{lem}\label{lem:independence-curve}Let $\nu_1\colon C_1\hookrightarrow X'$ and $\nu_2\colon C_2 \hookrightarrow X'$ be two distinct good curves. Let $\alpha$ be a $0$-cycle with modulus on $X'$, supported on $(C_1\setminus D)\cap (C_2\setminus D)$ and let $n$ be as above an integer prime to the characteristic of $k$. 
	    Then we have an equality in $\CH_0(X'|D)$ 
\[n_{C_1}^{-1}(\alpha) = n_{C_2}^{-1} (\alpha).\]
\begin{proof}
First, we note that since $X$ is smooth and since $X'$ is obtained as blow-up along smooth centres, $X'$ is a smooth $k$-variety as well. We will use a pencil argument, realizing $C_1$ and $C_2$ as components of the two special fiber of a $\mathbb{P}^1$-family of curves in $X'$. 
For this reason, we can introduce an intermediate player: let $C'$ be a general complete intersection of hypersurfaces containing the support of $\alpha$ and intersecting $C_1$ transversally at each point of $C_1\cap C'$. It is clear, by the symmetry between $C_1$ and $C_2$, that it would be enough to show the equality $n_{C_1}^{-1}(\alpha) = n_{C'}^{-1} (\alpha)$ in  $\CH_0(X'|D)$. 

We can then replace $C_2$ with $C'$ and continue the proof. 
Arguing as in the proof of Proposition \ref{prop:key-prop}, we can assume that $X'$ is a surface admitting a projective model $\ol{X'}$ such that $\ol{X'}\setminus \ol{D}$ is normal and such that, for $i=1,2$, $\nu_{C_i}\colon C_i\to X'$ admits a compactification $\ol{\nu_{C_i}} \colon \ol{C_i}\to \ol{X'}$. Using Lemma \ref{lem:good-compactification}, (that does not require the curve $C$ to be integral), we can assume that $\ol{C_i}\cap \ol{D} = C_i \cap D$ for $i=1,2$, so that the choice of the compactification does not effect the modulus condition of the curves. Again as in the proof of Proposition \ref{prop:key-prop}, by solving the singularities of $\ol{X'}$ that are on $\ol{C_i}\setminus C_i$ and the singularities of $\ol{C_i}\setminus C_i$, we can assume that every point $x\in \ol{C_i}\setminus C_i$ is a regular point of $\ol{X'}$ and that $\ol{C_i}$ is also regular there. 

We keep using the notations of Proposition \ref{prop:key-prop}. Let $P$ be the pencil $P = \{W_t | t\in \mathbb{P}^1\}$ of hyperplane sections of $\ol{X'}$ interpolating between $C_1$ and $C_2$. More precisely, we require that 
\begin{romanlist}
\item The generic member $W_t$ is integral and misses the singular locus of $\ol{D}_{red}$. The restriction $W_t\cap X'$ is regular.
\item The base locus of $P$ contains the support of $\alpha$ and misses $\ol{D}$. The rational map $\ol{X'}\dashrightarrow \mathbb{P}^1$ determined by $P$ becomes a morphism after a single blow-up of each point in the base locus (this is achieved by the condition that $C_1$ and $C_2$ intersect transversally). 
\item The special fibers are connected reduced curves of the form $W_0 = C_1 \cup E_1$ and $W_\infty = C_2 \cup E_2$, where $E_i$ are  integral curves, intersect $\ol{D}$ properly and away from $\ol{X'} \setminus X$ and are disjoint from the base locus of $P$. 
\end{romanlist}
The proof of Proposition 2.34 in \cite{tor-div-rec} now goes through: we give a sketch of the argument for completeness. Write $W_P$ for the blow-up of $\ol{X'}$ at every point of the base locus of $P$ and write $Z$ for the divisor on $W_P$ determined by the cycle $\alpha$. Write $u\colon W_P\to \ol{X'}$ for the blow-down map. The pencil $W_S = W_P\times_{\mathbb{P}^1} S\to S$ for $S$ the spectrum of the local ring of $\mathbb{P}^1$ at the origin is a flat projective family with constant arithmetic genus, making the relative Picard functor $\mathbf{Pic}^0_{W_S|D_S}\to S$ representable (as we saw in the proof of Proposition \ref{prop:key-prop}). Using the $n$-divisibility of generalized Jacobians over algebraically closed fields, we can find a DVR $S'\to S$ dominating $S$ so that there is a section
\[\gamma'\colon S' \to \mathbf{Pic}^0_{W_{S'}|D_{S'}} \]
such that $ n \gamma'(s') = Z_s$ and such that $n \gamma'(\ol{\eta'}) = Z_{\ol{\eta}}$, where $s'$ dominates $0\in \mathbb{P}^1$, $\ol{\eta}$ (resp.~$\ol{\eta'}$) is a geometric generic point of $S$ (resp.~$S'$). Let $Z'$ be the horizontal Cartier divisor representing $\gamma'$. Note that we can assume that $Z'$ is supported on $X'\setminus D$ (in particular, away from $\ol{X'}\setminus X$). This defines a map
\[Z'\colon S'\to \CH_0(X'|D)\]
satisfying $n Z'(s') = u_*(Z_s) =\alpha $ and $n Z'(\ol{\eta'}) = u_*(Z_{\ol{\eta'}}) = \alpha $ in $\CH_0(X'_{\ol{\eta'}}|D_{\ol{\eta'}})$. In particular, the family of cycles $Z' - Z'(s')$ parametrized by $S'$ defines a family of $n$-torsion cycles in $\CH_0(X'_{\ol{\eta'}}|D_{\ol{\eta'}})$, which is constant by the rigidity Theorem \ref{thm:rigidity}. Hence we have an equality 
\[n^{-1}_{C_1}(\alpha) = n^{-1}_{W_{\ol{\eta '}}}(\alpha).\]
Replacing $W_0$ with $W_\infty$ gives $n^{-1}_{C_2}(\alpha) = n^{-1}_{W_{\ol{\eta '}}}(\alpha)$, completing the proof.
\end{proof}
\end{lem}
\begin{rmk}\label{rmk:can-use-X-instead-Xprime} We note that the proof of Lemma \ref{lem:independence-curve} works if one replaces $X'$ with $X$.
\end{rmk}
\subsubsection{}\label{sssec:def-maps-n-1-after-independence} We resume the notations of Lemma \ref{lem:lem-blow-up}. By Lemma \ref{lem:independence-curve} and Remark \ref{rmk:can-use-X-instead-Xprime}, we  can construct well defined maps
\[n_{X'}^{-1}\colon Z_0(X'\setminus D) \to \CH_0(X'|D) \]
\[n_{X}^{-1}\colon Z_0(X\setminus D) \to \CH_0(X|D) \]
as follows. Since both $X$ and $X'$ are smooth, given any zero cycle $\alpha$ supported on $X\setminus D$ or on $X'\setminus D$, we can find an integral curve $C$, containing $\alpha$ and smooth along $\alpha$. In fact, we can assume that such curve $C$ is good in the sense of the definition given in \ref{def:definition-n-1}. 

For such $C$, we set $n^{-1}_X (\alpha)$ (or $n^{-1}_{X'}(\alpha)$) to be the class in $\CH_0(X|D)$ (resp. in $\CH_0(X'|D)$) of $n^{-1}_{C}(\alpha)$. By Lemma \ref{lem:independence-curve} and Remark \ref{rmk:can-use-X-instead-Xprime}, this is well defined and does not depend on the chosen good curve $\alpha$.
By construction, the maps $n_X^{-1}$  and $n^{-1}_{X'}$ satisfy $n_X^{-1}(n \alpha) = \alpha$ and $n_{X'}^{-1}(n\alpha) = \alpha$. Note that since given any two $0$-cycles $\alpha_1$ and $\alpha_2$ with modulus $D$ we can always find a good curve in $X$ (or in $X'$) containing the union of their supports, the maps $n_X^{-1}$ and $n_{X'}^{-1}$ are group homomorphisms. 


\subsubsection{}The proof of the following Lemma is identical to the corresponding statement in \cite{tor-div-rec}, Lemma 2.23 (taken from \cite[Lemma 3.2]{MarcTorsion}).
\begin{lem} Let $u\colon X'\to X$ be a blow-up at a point $x\in X\setminus D$. Then the following diagram commutes:
\[\xymatrix{ Z_0(X'\setminus D) \ar[r]^{n_{X'}^{-1}} \ar[d]_{u_*} & \CH_0(X'|D) \ar[d]_{u_*} \\ 
Z_0(X\setminus D) \ar[r]^{n_{X}^{-1}} & \CH_0(X|D) 
}
\]
\end{lem}
\begin{lem}\label{lem:n-1-factors-through-rat-equiv} The map $n_{X}^{-1}$ factors through $\CH_0(X|D)$. 
\begin{proof}
Let $\nu\colon C\to X$ be a finite map from a normal curve $C$ to $X$ such that $\nu(C)$ is birational to $C$ and is in good position with respect to $D$. Let $f\in k(C)^\times$ be a rational function on $C$ that is congruent to $1$ modulo $\nu^{*}D$. We have to show that $n_{X}^{-1}(\nu_*({\rm div} f)) = 0$ in $\CH_0(X|D)$. Let $u\colon X' \to X$ be a sequence of blow-ups of $X$ at the singular points of $\nu(C)$ away from $D$. By the previous Lemma, we
have
\[  n_{X}^{-1}(\nu_*({\rm div} f))  = u_*(n_{X'}^{-1} ( \nu'_*({\rm div} f)) )\]
where we have a factorization $C\to C'\hookrightarrow X'$, where $\nu'\colon C'\to X'$ is the strict transform of $\nu(C)$ in $X'$. Note that $C'$ is regular away from $D$. It is then enough to show that $n_{X'}^{-1} ( \nu'_*({\rm div} f)) = 0$. But we have $n_{X'}^{-1} ( \nu'_*({\rm div} f)) = n^{-1}_{C'}({\rm div}(f))$, since $n_{X'}^{-1}$ of any cycle can be computed by choosing a good curve containing it, thanks to Lemma \ref{lem:independence-curve}, and the latter is clearly zero.
\end{proof}
\end{lem}
We are finally ready to state and prove our main Theorem
\begin{thm}[See \ref{thm:Main-Theorem-Intro}]\label{thm:Main-Theorem}Let $X$ be a smooth affine $k$-variety of dimension at least $2$, $D$ an effective Cartier divisor on it. Then the Chow group of $0$-cycles on $X$ with modulus $D$, $\CH_0(X|D)$, is torsion free, except possibly for $p$-torsion if the characteristic of $k$ is $p>0$.
    \begin{proof} Thanks to Lemma \ref{lem:n-1-factors-through-rat-equiv}, we have for every $n$ prime to the characteristic of $k$, a well defined group homomorphism
        \[n_X^{-1}\colon \CH_0(X|D)\to \CH_0(X|D)\]
        and this is, by the remarks in \ref{sssec:def-maps-n-1-after-independence}, inverse to the multiplication by $n$, proving the claim.
        \end{proof}
    \end{thm}
\begin{rmk} One should note that Theorem \ref{thm:Main-Theorem} can not be directly pushed forward to encompass the $p$-torsion part of the Chow groups. For this, we refer the reader to \cite{BK}.
\end{rmk}    
    
\bigskip

\noindent{\bf Acknowledgments.} I wish to thank Marc Levine for much advice and for pointing out a gap in an earlier version of this note,  as well as for the financial support via the Alexander von Humboldt foundation and the DFG Schwerpunkt Programme 1786 "Homotopy theory and Algebraic Geometry".
I learned about many ideas on cycles on singular varieties from Amalendu Krishna during my joint work with him. I would like to thank him cordially for that prolific collaboration. I'm also grateful to J.~Cao, W.~Kai and R.~Sugiyama for a careful reading of this manuscript. 

I sincerely appreciate the referee's comments to an earlier draft of this paper, and for his/her efforts in improving and clarifying the exposition. 

\bibliography{bib} 

\begin{thebibliography}{10}

\bibitem{Barbieri-Viale:2010aa}
{\sc L.~Barbieri-Viale and B.~Kahn}, {\em On the derived category of
  1-motives}, vol.~381 of Ast\'erisque, Soci{\'e}t{\'e} Math{\'e}matique de
  France, 2016.

\bibitem{BThesis}
{\sc F.~Binda}, {\em Motives and algebraic cycles with moduli conditions}, PhD
  thesis, University of Duisburg-Essen, 2016.

\bibitem{tor-div-rec}
{\sc F.~Binda, J.~Cao, W.~Kai, and R.~Sugiyama}, {\em Torsion and divisibility
  for reciprocity sheaves and 0-cycles with modulus}, J. Algebra, 469 (2017),
  pp.~437--463.

\bibitem{BK}
{\sc F.~Binda and A.~Krishna}, {\em Zero cycles with modulus and zero cycles on
  singular varieties},  (2017).
\newblock arXiv:1512.04847v3 [math.AG].

\bibitem{NeronModels}
{\sc S.~Bosch, W.~L{\"u}tkebohmert, and M.~Raynaud}, {\em N\'eron models},
  vol.~21 of Ergebnisse der Mathematik und ihrer Grenzgebiete (3) [Results in
  Mathematics and Related Areas (3)], Springer-Verlag, Berlin, 1990.

\bibitem{MR3391879}
{\sc T.~Geisser}, {\em Rojtman's theorem for normal schemes}, Math. Res. Lett.,
  22 (2015), pp.~1129--1144.

\bibitem{GW}
{\sc U.~G\"ortz and T.~Wedhorn}, {\em Algebraic geometry {I}}, Advanced
  Lectures in Mathematics, Vieweg + Teubner, Wiesbaden, 2010.
\newblock Schemes with examples and exercises.

\bibitem{Gro}
{\sc A.~Grothendieck}, {\em {Techniques de construction en g\'eom\'etrie
  analytique. IV: Formalisme g\'en\'eral des foncteurs repr\'esentables}},
  {Familles d'Espaces Complexes et Fondements de la Geom. Anal., Sem. Henri
  Cartan (1960/61), No.1}, Tome 13 (1962), pp.~1--28.

\bibitem{KSY}
{\sc B.~Kahn, S.~Saito, T.~Yamazaki, and K.~R{\"u}lling}, {\em Reciprocity
  sheaves}, Compositio Mathematica, 152 (2016), pp.~1851---1898.

\bibitem{KS1}
{\sc M.~Kerz and S.~Saito}, {\em Chow group of $0$ -cycles with modulus and
  higher-dimensional class field theory}, Duke Math. J., 165 (2016),
  pp.~2811--2897.

\bibitem{AK}
{\sc S.~L. Kleiman and A.~B. Altman}, {\em Bertini theorems for hypersurface
  sections containing a subscheme}, Comm. Algebra, 7 (1979), pp.~775--790.

\bibitem{Krishna:2015aa}
{\sc A.~Krishna}, {\em Zero cycles on affine varieties},  (2015).
\newblock arXiv:1511.04221v1 [math.AG].

\bibitem{MarcBookUnpublished}
{\sc M.~Levine}, {\em A geometric theory of the {C}how ring for singular
  varieties}.
\newblock Unpublished manuscript, 1985.

\bibitem{MarcTorsion}
\leavevmode\vrule height 2pt depth -1.6pt width 23pt, {\em Torsion zero-cycles
  on singular varieties}, Amer. J. Math., 107 (1985), pp.~737--757.

\bibitem{Levine3}
\leavevmode\vrule height 2pt depth -1.6pt width 23pt, {\em Zero-cycles and
  {$K$}-theory on singular varieties}, in Algebraic geometry, {B}owdoin, 1985
  ({B}runswick, {M}aine, 1985), vol.~46 of Proc. Sympos. Pure Math., Amer.
  Math. Soc., Providence, RI, 1987, pp.~451--462.

\bibitem{LW}
{\sc M.~Levine and C.~Weibel}, {\em Zero cycles and complete intersections on
  singular varieties}, J. Reine Angew. Math., 359 (1985), pp.~106--120.

\bibitem{Liu}
{\sc Q.~Liu}, {\em Algebraic geometry and arithmetic curves}, vol.~6 of Oxford
  Graduate Texts in Mathematics, Oxford University Press, Oxford, 2002.
\newblock Translated from the French by Reinie Ern{\'e}, Oxford Science
  Publications.

\bibitem{MilneTorsion}
{\sc J.~S. Milne}, {\em Zero cycles on algebraic varieties in nonzero
  characteristic: {R}ojtman's theorem}, Compositio Math., 47 (1982),
  pp.~271--287.

\bibitem{RaynaudRep}
{\sc M.~Raynaud}, {\em Sp\'ecialisation du foncteur de {P}icard}, Inst. Hautes
  \'Etudes Sci. Publ. Math.,  (1970), pp.~27--76.

\bibitem{Rojtman}
{\sc A.~A. Rojtman}, {\em The torsion of the group of {$0$}-cycles modulo
  rational equivalence}, Ann. of Math. (2), 111 (1980), pp.~553--569.

\bibitem{SerreGACC}
{\sc J.-P. Serre}, {\em Algebraic groups and class fields}, vol.~117 of
  Graduate Texts in Mathematics, Springer-Verlag, New York, 1988.
\newblock Translated from the French.

\bibitem{SS}
{\sc M.~Spie{\ss} and T.~Szamuely}, {\em On the {A}lbanese map for smooth
  quasi-projective varieties}, Math. Ann., 325 (2003), pp.~1--17.

\bibitem{MR989909}
{\sc V.~Srinivas}, {\em Torsion {$0$}-cycles on affine varieties in
  characteristic {$p$}}, J. Algebra, 120 (1989), pp.~428--432.

\bibitem{SrinivasOverview}
{\sc V.~Srinivas}, {\em Some applications of algebraic cycles to affine
  algebraic geometry}, in Algebraic cycles, sheaves, shtukas, and moduli,
  Trends Math., Birkh\"auser, Basel, 2008, pp.~185--215.

\bibitem{SV}
{\sc A.~Suslin and V.~Voevodsky}, {\em Singular homology of abstract algebraic
  varieties}, Invent. Math., 123 (1996), pp.~61--94.

\end{thebibliography}
\bibliographystyle{siam}

\end{document}